\def\marker{\>\hbox{${\vcenter{\vbox{
    \hrule height 0.4pt\hbox{\vrule width 0.4pt height 6pt
    \kern6pt\vrule width 0.4pt}\hrule height 0.4pt}}}$}\>}

\documentclass[12pt]{article}

\usepackage[left=2cm,right=2cm,top=3cm,bottom=3cm]{geometry}

\usepackage{amsmath, amssymb, latexsym, amsthm}
\usepackage[section]{placeins}

\usepackage{graphicx}
\usepackage{fullpage}

\usepackage{yhmath}
\usepackage{mathdots}
\usepackage{MnSymbol}%

\usepackage{tikz}
\usetikzlibrary{shapes}

\newtheorem{theorem}{Theorem} 
\newtheorem{theorem*}{Theorem} 

\newtheorem{lemma}[theorem]{Lemma}

\theoremstyle{definition}

\newtheorem{question}{Question}
\newtheorem{problem}{Problem}

\theoremstyle{remark}

{\end{enumerate}}

\usepackage[hang,flushmargin]{footmisc}

\title{Characterization of Graphs with Villainy  2}
\author{Sogol Jahanbekam \footnotemark[1] and Meng-Ru Lin \footnotemark[2]}

\date{}

\begin{document}

\maketitle

\begin{abstract}

Let $f$ be an optimal proper coloring of a graph $G$ and let $c$ be a coloring of the vertices of  $G$ obtained by permuting the colors on vertices in the proper coloring $f$.  The villainy of $c$, written $B(c)$, is the minimum number of vertices that must be recolored to obtain a proper coloring of $G$ with the additional condition that the number of times each color is used does not change.  The villainy of $G$ is defined as $B(G)=max_{c}B(c)$, over all optimal proper colorings of $G$. In this paper, we characterize  graphs $G$ with $B(G)=2$.

\noindent{\bf Keywords: proper coloring,  	structural characterization of families of graphs, 	graph labelling, 05C15, 05C75, 05C78}
\end{abstract}

\renewcommand{\thefootnote}{\fnsymbol{footnote}}
\footnotetext[1]{
Department of Mathematics and Statistics, San Jose State University, San Jose, CA; {\tt sogol.jahanbekam@sjsu.edu.}\\{Research was supported in part by NSF grant CMMI-1727743.}
}

\footnotetext[2]{
 Department of Mathematics and Statistics, San Jose State University, San Jose, CA;
{\tt meng-ru.lin@sjsu.edu.} }
\renewcommand{\thefootnote}{\arabic{footnote}}

\baselineskip18pt

\section{Introduction}

A coloring over the vertices of a graph $G$ is \textit{proper} if adjacent vertices receive different colors.
 A \textit{proper $k$-coloring} of a graph $G$ is a coloring of the vertices of $G$ using $k$ colors. The \textit{chromatic number} of a graph $G$, denoted $\chi(G)$, is the smallest integer $k$ such that a proper $k$-coloring of $G$ exists.

Let $G$ be a graph with $\chi(G)=k$ and let 
  $c$ be a coloring of the vertices of  $G$ that is a permutation of a proper $k$-coloring of $G$. The \textit{weak  villainy} of $c$, denoted $B_w(c)$,  is the minimum number of vertices that must be recolored with the same set of colors as $c$ in order to obtain a proper coloring. The \textit{villainy} of $c$, denoted $B(c)$, is the minimum number of vertices that must be recolored with the same set of colors as $c$ with the additional condition that each color must appear exactly as many times as it does in $c$. The \textit{villainy} or \textit{weak villainy} of the graph $G$ are defined as:
 
 $$B(G)=\max\{B(c):\ c\in P(G)\}$$

  $$B_w(G)=\max\{B_w(c):\ c\in P(G)\},$$

respectively, where $P(G)$ is the set of all proper colorings of $G$ with $\chi(G)$ colors.

Note that finding the chromatic number of graphs is an NP-complete problem \cite{K}.
Studying villainy and weak villainy of graphs can lead to new algorithms for graph coloring problems. For example, if an efficient upper-bound for the villainy of graphs is found, then it could help us to navigate from a random coloring on $V(G)$ to a proper coloring of $G$ efficiently or argue  that such proper coloring does not exist. Villainy of graphs is also closely related to entropy. In information theory, the \textit{entropy} (also known as Shannon entrophy)  of a random variable is the average level of ``surprise"  or ``uncertainty" inherent in the variable's possible outcomes (see  \cite{S}). Considering a random permutation of a proper coloring of a graph $G$, the villainy of $G$ would represent the number of surprises we see in the permutation.

 Clark et al. \cite{C} first introduced the topic. Among other results, they proved that $B(G)=0$ if and only if $B_w(G)=0$ if and only if  $G$ is a complete graph or an empy graph if and only if $B(G)\leq 1$. They also showed that $B_w(G)=1$ if and only if either $|V(G)|\geq 3$ and $E(G)$ consists of one edge or two incident edges, or
$|V(G)|\geq 4$ and $G$ consists of a complete graph plus an isolated vertex or a pendant vertex.

  They also determined the villainy and weak villainy of connected bipartite graphs:
 
 \begin{theorem}[\cite{C}]\label{bipartite}
Suppose $G$ is a connected bipartite graph of order $n$ with $n\geq 3$. Let $X$ and $Y$ be bipartite sets of $G$ with $x=|X|$ and $n-x=|Y|$, where $|X|\leq |Y|$.  If $x<(n-x)$, then $B(G)=2x$. Otherwise, $B(G)=2\lceil \frac{n}{4}\rceil$. We also have $B_w(G)=min(2x, 2\lfloor\frac{n}{4}\rfloor)$ unless $n\equiv 3$ (mod 4), in which case $B_w(G)=min(2x, 2\lfloor\frac{n}{4}\rfloor+1)$.
 \end{theorem}


A number of open questions on villainy and weak villainy of graphs were raised in \cite{WH} as follows:
 
 \begin{problem}
Characterize the graphs $G$ such that $B(G)=2$. What is the least $k$ such that it is NP-complete to determine if $B(G)=k$?
 \end{problem}
 
 \begin{question}
 Is it the case that the villainy of a cycle with $2k+1$ vertices is $k$ when $k \geq 2$? (Note: it was proved in \cite{C} that $B_w(C_{2k+1})=k$ when $k\geq 2$.)
 \end{question}
 
  \begin{question}
 What is the villainy and weak villainy of complete multipartite graphs?
 \end{question}
 
  \begin{question}
  What are the largest possible values of $B(G)$ and $B_w(G)$ when $G$ has $n$ vertices and $\chi(G)=k$?
 \end{question}

\begin{question}
Let $G+H$ denote the disjoint union of graphs $G$ and $H$. Clark et. al \cite{C} proved that $B_w(G+H)\geq B_w(G)+B_w(H)$ when $\chi(G)=\chi(H)$. Does the inequality always hold? When does equality hold?
\end{question}

Here we address Problem 1 by classifying the graphs with
 villainy  $2$:

\begin{theorem}\label{main}
If $G$ is a graph with $B(G)=2$, then $G$ is one of the following graphs:
\begin{enumerate}
\item $G=K_{1,t}$, for some $t\geq 2$
\item $G$ is a 6-vertex connected bipartite graph with parts of size 3
\item $G=P_4$ or $G=P_4+K_1$ or $G=P_4+K_2$
\item $G=C_4$ or $G=C_4+K_1$ or $G=C_4+K_2$
\item $G=C_5$
\item $G=K_2+K_2$ or $G=K_2+K_2+K_2$
\item $G=K_3+K_2$
\item $G=K_3+K_3$
\item $G=K_2+K_2+K_1$
\item $n(G)\geq 4$ and $G$ has a vertex $v$ of degree 0 or 1 such that $G-v$ is a complete graph
\item $G=P_3+rK_1$, where $r\geq 1$
\item $G=K_2+rK_1$, where $r\geq 1$
\end{enumerate}
\end{theorem}

\section{Tools}

In this section, we list additional definitions and some known results on proper coloring of graphs, which we will use during the proof of Lemmas in Section 3.

We denote the order of a graph $G$ by $n(G)$. We also denote the disjoint union of graphs $G$ and $H$ by $G+H$. We represent a complete graph, a cycle, and a path of order $n$ with $K_n$, $P_n$, and $C_n$, respectively. $K_{m,n}$ is a complete bipartite graph with parts of size $m$ and $n$, respectively.

We say we have a \textit{subdivision} of an edge $uv$ in a graph $G$ if we introduce a new vertex $w$ in $G$ and replace the edge $uv$ by edges $uw$ and $wv$.
A \textit{subdivision} of a graph $G$ is a graph obtained from $G$ by repeated edge subdivisions. An oddly subdivision of $K_4$  is a subdivision of $K_4$ (the complete graph with four
vertices) such that all four cycles corresponding to triangles in $K_4$ are odd. Catlin \cite{CA} proved the following:

\begin{theorem}(\cite{CA})\label{subd}
Every graph with no oddly subdivision of $K_4$ is 3-colorable.
\end{theorem}

We also use the following result during the proof of Lemma \ref{chrom4}:

\begin{theorem}(\cite{CH})\label{grot}
The smallest triangle-free graph with chromatic number 4 has 11 vertices. 
\end{theorem}

The graph described in Theorem \ref{grot} is known as Gr\"otzsch graph.

\section{Properties of graphs with villainy 2}

In this section, we present a sequence of lemmas  which will be our tool to prove Theorem \ref{main} in Section 3.

\begin{lemma}\label{triangle}
Let $G$ be a graph containing a triangle $abc$. If $G$ has a proper coloring using $\chi(G)$ colors in such a way that at least one color class has size at least 3, then $B(G)\geq 4$.
\end{lemma}

\begin{proof}
Let $c$ be a proper coloring  of $G$ that uses $\chi(G)$ colors in such a way that a color class  has size at least 3. Let $c^*$ be any permutation of $c$ over the vertices of $G$ in which the vertices of the triangle receive the same color, say 1. To reconstruct a proper coloring of $G$, we need to change the color of at least two vertices of the triangle. Moreover, since the new colors of these two vertices of the triangle are different from 1, in order to keep the number of times each color appears, we need to change the color of at least two vertices outside the triangle as well. Therefore $B(G)\geq 4$. 
\end{proof}

\begin{lemma}\label{color-class-size-4}
Let $G$ be a graph with $\chi(G)\geq 3$. If $G$ has a proper coloring using $\chi(G)$ colors in such a way that at least one color class has size at least 4, then $B(G)\geq 4$.
\end{lemma}

\begin{proof}
Let $c$ be a proper coloring  of $G$ that uses $\chi(G)$ colors in such a way that a color class  has size at least 4. If $G$ has a triangle, then Lemma \ref{triangle} implies $B(G)\geq 4$, as desired. Hence, suppose $G$ has no triangle. Since moreover $\chi(G)\geq 3$, the graph $G$ contains an odd cycle of length at least 5. Therefore $G$ contains a path $uvwz$.  Let $c^*$ be any permutation of $c$ over the vertices of $G$ in which the vertices of the path $uvwz$ receive the same color. To restore a proper coloring of $G$, we need to change the color of at least two vertices of the path. Moreover,  in order to keep the number of times each color appears, we need to change the color of at least two vertices outside the path as well. Therefore $B(G)\geq 4$. 
\end{proof}

\begin{lemma}\label{diamond}
Let $G$ be a graph having a triangle one of whose vertices have degree at least 3 in $G$. If $G$ has a proper coloring using $\chi(G)$ colors in such a way that at least two color classes have size at least 2, then $B(G)\geq 3$.
\end{lemma}

\begin{proof}
Suppose $uvw$ is a triangle in $G$ and suppose $z$ is a vertex in $G$ with $vz\in E(G)$. Let $c$ be a proper coloring  of $G$ that uses $\chi(G)$ colors in $\{1,\ldots,\chi(G)\}$ in such a way that each of the colors 1 and 2  appears at least twice in $c$.   Let $c^*$ be any permutation of $c$ over the vertices of $G$ in which $c^*(u)=c^*(w)=1$ and $c^*(v)=c^*(z)=2$.  To reconstruct a proper coloring of $G$, we need to change the color of at least one vertex in $\{u,w\}$, and we need to change the color of at least one vertex in $\{v,z\}$. Moreover, a color different from 1 and 2 has to be assigned to a vertex in $\{u,v,w,z\}$.  As a result,  in order to keep the number of times each color appears, we need to change the color of at least one vertex outside $\{u,v,w,z\}$. Therefore $B(G)\geq 3$. 
\end{proof}

\begin{lemma}\label{matching}
Let $G$ be a graph having a matching of size at least 3. If $G$ has a proper coloring using $\chi(G)$ colors in such a way that at least three color classes have size at least 2, then $B(G)\geq 3$.
\end{lemma}

\begin{proof}
Let $c$ be a proper coloring  of $G$ that uses $\chi(G)$ colors in such a way that at least three color classes  have size at least 2. Let $c^*$ be any permutation of $c$ over the vertices of $G$ in which the endpoints of three disjoint edges in $G$ receive the same color. To reconstruct a proper coloring of $G$, we need to change the color of at least one endpoint of each of these three edges.  Therefore $B(G)\geq 3$. 
\end{proof}

\begin{lemma}\label{chi(G)=3-n(G)7+}
Let $G$ be a graph of order at least 7 with $\chi(G)=3$. We have $B(G)\geq 3$. 
\end{lemma}

\begin{proof}
Let $c$ be a proper coloring  of $G$ that uses 3 colors. If $c$ has a color class of size at least 4, then by Lemma \ref{color-class-size-4} we have $B(G)\geq 4$, as desired. Therefore suppose $c$ has no color class of size at least 4. Since $G$ has at least 7 vertices and $\chi(G)=3$, at least one color class in $G$ has size at least 3. Therefore by Lemma \ref{triangle} we may suppose that $G$ has no triangle.  

If $G$ contains $C_5$, then we have $B(G)\geq 4$, because any permutation of $c$ over the vertices of $G$ in which the vertices of the $C_5$ receive colors 1,1,1,2,2, respectively, requires at least 4 recolorings  to reconstruct a proper coloring of $G$. Hence, suppose $G$ has no $C_5$. 

Since $\chi(G)=3$ and $G$ has no triangle and no $C_5$, it must contain an odd cycle of length at least 7. Hence, $G$ contains a matching of size at least 3. If each color class of $c$ has size at least 2, then by Lemma \ref{matching} we have $B(G)\geq 3$. Therefore suppose $G$ has a color class of size 1, which implies $G$ has exactly 7 vertices. Since moreover $G$ has no triangle and no $C_5$, we have $G=C_7$. Therefore $G$ has a proper coloring using 3 colors in such a way that each color class has size at least 2, which implies $B(G)\geq 3$, according to Lemma \ref{matching}.
\end{proof}

\begin{lemma}\label{two triangles}
Let $G$ be a graph of order 6 with $\chi(G)=3$. We have $B(G)\geq 3$ unless $G$ has two components each of which is a triangle.
\end{lemma}

\begin{proof}
Let $c$ be a proper coloring  of $G$ that uses 3 colors. If $c$ has a color class of size at least 4, then by Lemma \ref{color-class-size-4} we have $B(G)\geq 4$, as desired. Therefore suppose $c$ has no color class of size at least 4.  As a result, either the color classes of $c$ have sizes $3,2,1$ or the color classes of $c$ have sizes $2,2,2$. 

First suppose the color classes of $c$ have sizes 3,2,1, where color 1 appears with multiplicity 3, color 2 appears with multiplicity 2, and color 3 appears with multiplicity 1. By Lemma \ref{triangle} we may suppose that $G$ has no triangle. Since $\chi(G)=3$ and $G$ has 6 vertices, the graph $G$ must contain a cycle of length 5. Any permutation of $c$ over the vertices of $G$ in which the vertices of the $C_5$ receive colors 1,1,1,2,2, respectively, requires at least 4 recolorings  to reconstruct a proper coloring of $G$. Hence in this case, we have $B(G)\geq 4$, as desired. 

Now suppose the color classes of $c$ have sizes 2,2,2. By Lemma \ref{matching} we may suppose that $G$ has no matching of size at least 3. If $G$ has a vertex whose removal makes the graph bipartite, then $G$ has a proper coloring with 3 colors in such a way that the multiplicities of colors are 3,2,1. In this case the above argument can be applied to show $B(G)\geq 3$. Hence, we may suppose that removing any vertex from $G$ does not result in a bipartite graph.  Therefore, $G$ contains at least two odd cycles.

Since $\chi(G)=3$, $G$ contains a triangle or a cycle of length 5. First suppose $G$ has a cycle of length 5. If the vertex out of the cycle $C_5$ is an isolated vertex, then the above case can be applied. If the vertex out of the cycle $C_5$ has an edge to a vertex in $C_5$, then $G$ has a matching of size 3, which implies $B(G)\geq 3$ by Lemma \ref{matching}. Hence, suppose $G$ has no $C_5$. Therefore $G$ contains two triangles. 

By Lemma \ref{diamond} we may suppose that the two triangles of $G$ are disjoint and $G$ has no additional edges besides the edges of these triangles, as desired.
\end{proof}

\begin{lemma}\label{n=5}
Let $G$ be a graph of order at most 5 with $\chi(G)=3$. We have $B(G)\geq 3$ unless $G\in \{C_3, C_5, C_3+K_2, C_3+K_1 \}$ or $G$ is a triangle along with a pendant edge.    
\end{lemma}

\begin{proof}
Let $c$ be a proper coloring  of $G$ that uses 3 colors. Since $\chi(G)=3$, the graph $G$ has at least 3 vertices. Moreover, if $G$ has exactly three vertices, we have $G=K_3$. Note that $B(K_3)=0$. Therefore we may suppose that $G$ has 4 or 5 vertices.

First suppose that the order of $G$ is 4. In this case, the multiplicities of the colors in $c$ are 2,1,1. Hence, $G$ is either a triangle along with an isolated vertex or a pendant edge, or $G$ is a diamond.

If $G$ is a diamond, then we have $B(G)=4$, because if we permute the coloring of $c$ on $V(G)$ in such a way that the two vertices of degree 3 get the same color, then we will have to recolor all the vertices in order to reinstate a proper coloring of $G$. 

Now suppose $G$ is a triangle along with an isolated vertex or a pendant edge. In any permutation of $c$, there are at most one pair of adjacent vertices having the same color. In this case, switching the colors of two vertices reinstates a proper coloring of $G$. Therefore in this case, we have $B(G)=2.$ 

Now suppose $G$ has order 5. Since $\chi(G)=3$, $G$ contains a triangle or a $C_5$. If $G$ contains $C_5$ but not $C_3$, then $G=C_5$. We already argued in earlier proofs that $B(C_5)=2$. Hence, suppose $G$ contains a triangle.

By Lemma \ref{triangle} we can suppose that $G$ has a proper 3-coloring where the multiplicities of the colors are $2,2,1$, respectively. By Lemma \ref{diamond} none of the vertices of the triangle in $G$ can have a neighbor outside the triangle. 

Therefore we have $G=K_3+K_2$, since otherwise one of the cases in Lemma \ref{triangle} or \ref{diamond} applies to $G$. But for $G=K_3+K_2$,  we have $B(G)=2$, because in any permutation of $c$ on the vertices of $G$ there is at most one pair in $K_3$ and one pair in $K_2$ having the same colors, and switching the color of one vertex from each of the pairs reinstates a proper coloring of $G$. Therefore in this case $B(G)=2$. 
\end{proof}

\begin{lemma}\label{chrom4}
Let $G$ be a graph with $\chi(G)\geq 4$. We have $B(G)\geq 3$, unless $G$ is a complete graph or $G$ has a vertex of degree at most 1 whose removal from $G$ results in a complete graph.
\end{lemma}

\begin{proof}
If $G$ is a complete graph, then $B(G)=0$. If $G$ is not a complete graph, but has a vertex $v$ of  degree at most 1 such that $G-v$ is a complete graph, then $B(G)=2$, because in any permutation of a proper coloring of $G$ that uses $\chi(G)$ colors, switching the color of one pair of vertices suffices to guarantee $v$ and a vertex in $V(G)-N(v)$ both get the same color. Hence, suppose $G$ is none of these two cases. We prove that $B(G)\geq 3$. 

Let $c$ be a proper coloring  of $G$ that uses $\chi(G)$ colors. First suppose $c$ has a color class of size 2 and the rest of color classes have size 1. Suppose $u$ and $v$ are the vertices with $c(u)=c(v)$, and we may suppose this color is 1. Since all other colors appear once in $c$, the graph $G-\{u,v\}$ is a complete graph. If each of $u$ and $v$ has a non-neighbor in $V(G)-\{u,v\}$, then we can give them colors in $\{2,\ldots,\chi(G)\}$ and obtain a proper coloring of $G$ using $\chi(G)-1$ colors, which is not accepted. Hence, we may suppose that $u$ is adjacent to all vertices in $V(G)-\{u,v\}$. Therefore $G-v$ is a complete graph, and by the choice of $G$, we have $d(v)\geq 2$.

 Suppose $w,w'$ are two neighbors of $v$ in $G$. Let $c^*$ be a permutation of $c$ in which the vertices $w$ and $w'$ have color 1. In order to reconstruct a proper coloring of $G$, the vertex $v$ has to receive color 1, since 1 is the only color with multiplicity 2 in $c$. Therefore we need to recolor vertices $w$ and $w'$ in $c^*$, and since both of $w$ and $w'$ have the same color in $c^*$, we need to recolor two more vertices in $G$ (including vertex $v$.) Therefore in this case, we have $B(G)\geq 4$. Hence, we may suppose that $\chi(G)\leq n-2$.

Since $\chi(G)\leq n-2$, $c$ has either a color class of size at least 3 or it has at least two color classes each of size 2.  Since $\chi(G)\geq 4$, by Theorem \ref{subd} $G$ has an oddly subdivision of $K_4$.

 First suppose $G$ contains $K_4$ as a subgraph. If $c$ has a color class of size at least 3, consider a permutation of $c$ in which three vertices of $K_4$ receive the same color. If $c$ has two color classes of size 2, consider a permutation of $c$ in which two vertices of $K_4$ receive the same color and the remaining two vertices also receive the same color. In  both of the cases, at least 4 recolorings are required to reconstruct a proper coloring of $G$. Therefore when $G$ has $K_4$ as a subgraph, we have $B(G)\geq 4$. Hence, suppose $G$ has no $K_4$ as a subgraph. 

By Theorem \ref{subd} the graph $G$ has an oddly subdivision of $K_4$, and since $G$ contains no $K_4$ as a subgraph, the graph $G$ contains a matching of size 3. If $c$ has at least three color classes each of size at least 2, then Lemma \ref{matching} implies $B(G)\geq 3$, as desired. Hence suppose $c$ contains at most two color classes of size at least 2. By Lemma \ref{color-class-size-4} we may suppose that all color classes in $c$ have size at most 3. 

Therefore $G$ contains a complete subgraph with $\chi(G)-2$ vertices, because $\chi(G)-2$ color classes in $c$ have size 1. Hence, if $G$ is triangle-free, then we must have $\chi(G)=4$.  By Theorem \ref{grot} the smallest triangle-free graph with $\chi(G)\geq 4$ has at least 11 vertices.  Any 11-vertex graph with a proper coloring with 4 colors has a color class of size at least 4, or has at least three color classes each of size at least 2, which are not accepted.  Therefore $G$ cannot be triangle-free. 

Hence, $G$ contains a triangle $abc$. If $G$ has a color class of size at least 3, then by Lemma \ref{triangle} we have $B(G)\geq 4$, as desired. Hence, all color classes of $G$ have size at most 2, and exactly two color classes have size 2 in $c$. If $abc$ is not a component of $G$, then by Lemma \ref{diamond} we have $B(G)\geq 4$, as desired. Therefore $abc$ must be a component of $G$. Since $\chi(G)\geq 4$, the graph $G$ must have other components as well, and as a result, $G$ has at least three color classes of size at least 2, which is not accepted. 
\end{proof}

\section{Proof of Theorem \ref{main}}

Let $G$ be a graph with $B(G)=2$. If $\chi(G)\geq 4$, then by Lemma \ref{chrom4} and the fact that complete graphs have villainy  0, we have that $G$ has a vertex $v$ of degree 0 or 1 such that $G-v$ is a complete graph.

Now suppose $\chi(G)=3$. By Lemma \ref{chi(G)=3-n(G)7+} we have $n\leq 6$. If $n=6$, then by Lemma \ref{two triangles} we have $G=K_3+K_3$. If $n=5$, then by Lemma \ref{n=5} we have $G=C_5$ or $G=K_3+K_2$.

If $n=4$, then by Lemma \ref{n=5}  we have that $G$ has a vertex $v$ of degree 0 or 1 such that $G-v$ is a triangle. 

Note that we cannot have $n\leq 3$, because triangle has villainy  0.

Note that $\chi(G)\neq 1$, as empty graphs have villainy  0. Hence the final case is when $\chi(G)=2$. 

If $G$ is connected, then by Theorem \ref{bipartite} the graph $G$ is a star of order at least 3, $G=C_4$, $G=P_4$, or $G$ is a bipartite graph with parts each of size 3. 

Hence, suppose $G$ is bipartite and disconnected. We consider two cases.

\textit{Case 1.} $G$ has at least two components of order at least 2.  If $G$ has a proper 2-coloring  where a color class has size at least 4, then  we can  rearrange that proper coloing in such a way  that the endpoints of two disjoint edges in $G$ all get the same color. As a result, we need to recolor at least 4 vertices in order to regain a proper coloring of $G$ keeping the color multiplicties, which implieas $B(G)\geq 4$. Hence, each color class of $G$ has size at most 3 and therefore $n(G)\leq 6$.

If $n(G)=6$, then $G=C_4+ K_2$ or $G=P_4+ K_2$ or $G=K_2+ K_2+ K_2$, because otherwise $G$ can have a 2-coloring with a color class of size 4. Each of the graphs $C_4+K_2, P_4+K_2,$ and $K_2+K_2+K_2$ has villainy  2, because in any permutation of a proper 2-coloring of these graphs at most two edges exist that have endpoints with the same colors and on these cases switching the colors of two vertices will reproduce a proper coloring of the graph. 

Hence suppose $n(G)\leq 5$. Note that $G$ cannot be $K_2+P_3$, because if in a rearrangement of a proper 2-coloring of $G$ the two vertices of $K_2$ and the middle vertex of $P_3$ get the same color, then we need to recolor 4 vertices to regain a proper coloring of $G$ with the same color multiplicities.

Therefore  we have $G=K_2+K_2$  or $G=K_2+K_2+K_1$.  Both of these graphs, have villainy  2 because in any rearrangement of a proper coloring of these two graphs switching the colors of two vertices suffice to reinstate a proper coloring of $G$ keeping multiplicities of the colors.

\textit{Case 2.} $G$ has only one component of order at least 2. Let $G_1$ be the component of $G$ of order at least 2. Hence, the rest of components of $G$ are isolated vertices. Note that $B(G_1)\leq B(G)$, because $G_1$ is a subgraph of $G$. Moreover, $G_1$ is bipartite. Hence, $B(G_1)=0$ or $B(G_1)=2$ (the villany number of all graphs are even, because we aim to keep the multiplicities of the colors). 

If $B(G_1)=0$, then $G_1=K_2$. As a result, $G=K_2+rK_1$, for some $r\geq 1$. We note that such graph $G$ indeed has villainy  2.

If $B(G_1)=2$, then by our eariler argument, $G_1$ is a star of order at least 3, $G_1=C_4$, $G_1=P_4$, or $G_1$ is a bipartite graph with parts each of size 3. 

If $G_1$ is a bipartite graphs with parts each of size 3, then using the assumption that $G$ is disconnected, we can rearrange a proper coloring of $G$ in such a way that two vertices in the first part of $G_1$ and two vertices in the second part of $G_1$ all get the same color. This way, we need to recolor at least 4 vertices to obtain a proper coloring of $G$ in such a way that the multiplicities of the colors stay the same, which contradicts $B(G)=2$. So this case is not possible.

If $G_1=C_4$, then we must have $G=C_4+K_1$, because if $G=C_4+rK_1$ for $r\geq 2$, then we can have a proper coloring of $G$ with a color class of size at least 4. Then we can rearrange that coloring in such a way that all vertices of $C_4$ have the same color. Now we require recoloring at least 4 vertices to get a proper coloring of $G$ with the same multiplicities, which contradicts $B(G)=2$. Moreover, we have $B(C_4+K_1)=2$.

A simlar argument as above implies that when $G_1=P_4$, we have $G=P_4+K_1$.

If $G_1=K_{1,t}$ for some $t\geq 3$, then $G=K_{1,t}+rK_1$ with $r\geq 1$. Consider a proper 2-coloring of $G$ where color 1 has multiplicity 2 and color 2 has multiplicity $t+r-1$. Rearrange that coloring in such a way that  color 1 appears on two vertices of degree 1. Then we need to recolor at least 4 vertices to regain a proper coloring of $G$ with the same multiplicities. This is not accepted as $B(G)=2$.

Hence, we only need to check if $G_1=K_{1,2}$ ($=P_3$) is an acceptable option. In fact for this case we can have $G=P_3+rK_1$ for any $r$ with $r\geq 1$ because we would need to recolor at most two vertices of $G$ in any rearrangement of a proper coloring of $G$ in order to construct a proper coloring of $G$.

\end{document}